\documentclass[12pt,reqno]{amsart}
\usepackage{amsmath,amsthm,amsfonts,latexsym,amssymb,amscd,color}
\newtheorem{lemma}{Lemma}[section]

\newtheorem{theorem}[lemma]{Theorem}
\newtheorem{proposition}[lemma]{Proposition}
\newtheorem{clm}[lemma]{Claim}
\theoremstyle{definition}
\newtheorem{definition}[lemma]{Definition}

\newtheorem{example}[lemma]{Example}

\newtheorem{observation}[lemma]{Observation}

\numberwithin{equation}{section}
\newcommand{\N}{\mathbb{N}}

\newcommand{\ann}{\mbox{\rm ann}}
\newcommand{\X}{\mathcal{X}(S)}

\newcommand{\w}{\wedge}
\newcommand{\wi}{\bigwedge_{i=1}^n}
\newcommand{\wii}{\bigwedge_{i=0}^{n-1}}

\newcommand{\Sii}{\sum_{i=0}^{n-1}}

\def\f{\varphi}

\def\c{\chi}

\newcommand{\Sub}{\mathsf{S}}
\newcommand{\Prod}{\mathsf{P}}

\let\phi=\varphi
\let\epsilon=\varepsilon

\let\hat=\widehat

\begin{document}

\title[Divisibility Theory and Distributivity]{Divisibility
  Theory of Commutative Rings\\
  and Ideal Distributivity}

\author{P. N. \'Anh}
\address[P. N. \'Anh]{R\'enyi Institute of Mathematics\\
Hungarian Academy of Sciences\\
1364 Budapest, Pf. 127\\
Hungary}
\email{anh.pham.ngoc@renyi.mta.hu}

\author{Keith A. Kearnes}
\address[Keith A. Kearnes]{Department of Mathematics\\
University of Colorado\\
Boulder, CO 80309-0395\\
USA}
\email{Kearnes@Colorado.EDU}

\author{\'Agnes Szendrei}
\address[\'Agnes Szendrei]{Department of Mathematics\\
University of Colorado\\
Boulder, CO 80309-0395\\
USA}
\email{Szendrei@Colorado.EDU}

\thanks{This material is based upon work supported by
  the National Research, Development and Innovation Office NKFIH
  K119934,
  the Vietnam Institute for Advanced Study in Mathematics (VIASM),
  the Vietnamese Institute of Mathematics,
  the National Science Foundation grant no.\ DMS 1500254,
  the Hungarian National Foundation for Scientific Research (OTKA)
  grant no.\ K115518, and
  the National Research, Development and Innovation Fund of Hungary (NKFI)
  grant no.\ K128042.}


\begin{abstract}
  We begin by investigating the class of commutative unital
  rings in which no two distinct elements divide
  the same elements. We prove that this class
  forms a finitely axiomatizable,
  relatively ideal distributive quasivariety,
  and it equals the quasivariety generated by
  the class of integral domains with trivial unit group.
  We end the paper by proving a representation theorem that
  provides more evidence to the conjecture that
  B\'ezout monoids describe exactly the monoids of finitely generated
  ideals of commutative unital rings
  with distributive ideal lattice.
\end{abstract}

\maketitle

\section{Introduction}
In this paper we discuss four perspectives
for studying the divisibility theory
of commutative unital rings.

The first perspective is that of multiplicative monoids.
If $R$ is a commutative unital ring, then
in the multiplicative monoid $\langle R; \cdot\rangle$
we say that ``$a$ divides $b$'' (written $a\mid b$)
provided there exists $c$ such that $ac=b$.
This relation of divisibility is a
reflexive, transitive relation on $R$,
i.e. a quasiorder on $R$, and we call the structure $R$,
under the relation of divisibility,
the divisibility quasiorder of $R$.

The second perspective identifies two elements
of $\langle R; \cdot\rangle$
if they divide the same elements. This identification
means that we are 
passing from the divisibility quasiorder to its
quotient partial order. Since $a$ and $b$
divide the same elements exactly when the principal
ideals $(a)$ and $(b)$ are equal, 
this second perspective involves
studying the poset ${\mathcal I}_{\textrm{princ}}(R)$
of principal ideals of $R$.
(We order the elements of ${\mathcal I}_{\textrm{princ}}(R)$
by inclusion, thereby reversing the
quotient order.)

The third perspective involves enlarging the poset
${\mathcal I}_{\textrm{princ}}(R)$ to the join
semilattice ${\mathcal I}_{\textrm{f.g.}}(R)$ 
of finitely generated ideals of $R$.

The fourth perspective involves enlarging 
${\mathcal I}_{\textrm{f.g.}}(R)$ still further
to the lattice ${\mathcal I}(R)$ 
of all ideals.

The passage from $\langle R; \cdot\rangle$ to 
${\mathcal I}_{\textrm{princ}}(R)$, involving the formation of
a quotient, results in a loss
of information. One sees this clearly
by noting that when $\mathbb E$ and $\mathbb F$
are fields we must have
${\mathcal I}_{\textrm{princ}}(\mathbb E)\cong {\mathcal I}_{\textrm{princ}}(\mathbb F)$,
yet we needn't have
$\langle \mathbb E; \cdot\rangle\cong \langle \mathbb F; \cdot\rangle$.
(These latter monoids needn't even have the same size.)
Our first main goal is to investigate
the situation where the passage
from $\langle R; \cdot\rangle$ to 
${\mathcal I}_{\textrm{princ}}(R)$ results in no loss
of information. Equivalently, we examine and characterize
the situation where the divisibility quasiorder of $R$
is already a partial order. We will find that the
class of commutative unital rings whose
divisibility quasiorder is a partial order
is a relatively ideal distributive
quasivariety. Moreover, we show
that this quasivariety is axiomatized
relative to the class of commutative rings by the
single quasi-identity $xyz=z\to yz=z$.
Finally, this quasivariety is exactly
the quasivariety generated by the class of integral
domains which have trivial unit group.

Our second main objective is a contribution to the representation of
B\'ezout monoids as divisibility theories
$\mathcal{I}_{\rm princ}(R)=\mathcal{I}_{\rm f.g}(R)$
of B\'ezout rings (i.e., commutative unital rings whose finitely generated
ideals are principal).
B\'ezout monoids (see Definition~\ref{basic}) were introduced in
\cite{B3} (see also \cite{amv}) as an abstract framework for studying
divisibility in B\'ezout rings, and more generally, in arithmetical rings
(i.e., commutative unital rings $R$ with distributive ideal lattice
$\mathcal{I}(R)$).
It is known that the divisibility theory $\mathcal{I}_{\rm f.g}(R)$
of every arithmetical ring $R$ is a B\'ezout monoid~\cite{amv}, and
it is conjectured that, conversely, every B\'ezout monoid can be
represented as the divisibility theory $\mathcal{I}_{\rm f.g}(R)$
of an arithmetical ring $R$. By a result of Anderson~\cite{and},
if this is the case, then in fact, every B\'ezout monoid can be
represented as the divisibility theory
$\mathcal{I}_{\rm princ}(R)=\mathcal{I}_{\rm f.g}(R)$
of a B\'ezout ring $R$.
The conjecture has been proved for semihereditary
B\'ezout monoids~\cite{as1} and B\'ezout monoids
which have only one minimal $m$-prime filter~\cite{as2}.
In this paper we will prove the conjecture for
another class of B\'ezout monoids, which we call \emph{Vasconcelos monoids},
due to the fact that the divisibility theory of the B\'ezout ring constructed by
Vasconcelos~\cite[Example~3.2]{vas} belongs to this class.

\section{The quasivariety of
  rings whose principal ideals have\\ unique generators}
\label{s:2}

Our goal in this section is to describe the class
of commutative rings whose principal ideals have unique generators.
The main result is that this class is a relatively ideal
distributive quasivariety,
so let us explain now what that means.
(For more details about relatively congruence distributive/modular
quasivarieties,
we refer to \cite{kearnes1, kearnes2, kearnes-mckenzie}.)

A \emph{quasi-identity} in the language of commutative rings
is a universally quantified implication of the form
\[
(s_1=t_1)\wedge \cdots \wedge (s_n=t_n)\to (s_0=t_0)
\]
where $s_i$ and $t_i$ are ring terms (= ``words'', or ``polynomials'').
We allow $n=0$, in which case the quasi-identity reduces
to an identity: $s_0=t_0$ (universally quantified).
To emphasize this last point:
identities are special quasi-identities.

A \emph{variety} is a class axiomatized by identities.
A \emph{quasi-variety} is a class axiomatized by quasi-identities.
For an example of the former, the class of commutative rings
is a variety. For an example of the latter, 
the class of rings axiomatized
by the identities defining commutative rings together with the
quasi-identity $(x^2=0)\to (x=0)$
is the quasivariety of reduced commutative rings
(rings with no nonzero nilpotent elements).

If $\mathcal Q$ is a quasivariety of commutative rings,
$R\in\mathcal Q$, and $I\lhd R$ is an ideal of $R$,
then $I$ is a \emph{$\mathcal Q$-ideal} (or a \emph{relative} ideal) if
$R/I\in {\mathcal Q}$.
For example, if $\mathcal Q$ is the quasivariety of
commutative reduced rings and $R\in\mathcal Q$,
then $I$ is a relative ideal of $R$ exactly when
$I$ is a semiprime ideal of $R$.

The collection of $\mathcal Q$-ideals of some $R\in\mathcal Q$, when
ordered by inclusion, forms
an algebraic lattice. It is not a sublattice of the
ordinary ideal lattice, but it is a sub\underline{set} of the
ordinary ideal lattice that is closed under arbitrary
meet.

A quasivariety $\mathcal Q$ of commutative rings is
relatively ideal distributive if the $\mathcal Q$-ideal
lattice of any member of $\mathbb Q$ satisfies the distributive law:
\[
I\wedge (J\vee K) = (I\wedge J)\vee (I\wedge K).
\]
Here, the meet operation is just intersection ($I\wedge J = I\cap J$)
while the join operation depends on $\mathcal Q$;
all that can be said is that $I\vee J$ is the least $\mathcal Q$-ideal
that contains $I\cup J$ (or, equivalently, contains $I+J$).

It is interesting to recognize that a quasivariety of rings
is relatively ideal distributive, because any distributive
algebraic lattice is isomorphic to the lattice of open
sets of a topology defined on the set of meet irreducible
lattice elements. This means that, if $\mathcal Q$ is relatively
ideal distributive, then to each member of $\mathcal Q$
there is a naturally associated topological space, 
its $\mathcal Q$-spectrum.
It is possible to treat a member $R\in \mathcal Q$
as a ring of functions defined over the spectrum of $R$.
It turns out that the quasivariety of commutative reduced
rings, mentioned earlier as an example, is relatively
ideal distributive, and for this $\mathcal Q$
the $\mathcal Q$-spectrum of any $R\in\mathcal Q$
is just the ordinary prime spectrum of $R$.

The main result of this section is that the
class of commutative rings whose principal ideals
have unique generators is a relatively ideal distributive
quasivariety, and for which we provide an axiomatization.

\begin{theorem}
\label{ak1}
Let $\mathcal Q$ be the class of all commutative rings $R$
(with $1$) having the property that 
each principal ideal has a unique generator.
Let $\mathcal D$ be the class of domains in $\mathcal Q$.
\begin{enumerate}
\item[(1)] 
$\mathcal Q$ is a quasivariety. It is exactly the
class of rings axiomatized by the quasi-identity 
$(xyz=z)\to (yz=z)$
along with the identities defining the variety of all commutative rings.
All rings in $\mathcal Q$ have trivial unit group and are reduced.
Such rings are $\mathbb F_2$-algebras.
\item[(2)]
$\mathcal D$ is 
exactly the class of domains with trivial unit group.
\item[(3)]
  $\mathcal Q$ consist of the subrings of products of members
  of $\mathcal D$ (we write ${\mathcal Q}=\Sub\Prod({\mathcal D})$).
\item[(4)]
${\mathcal Q}$ is a relatively ideal distributive quasivariety.
\item[(5)]
The class of locally finite algebras in ${\mathcal Q}$ is 
the class of Boolean rings. This class is the largest subvariety
of $\mathcal Q$.
\end{enumerate}
\end{theorem}

\begin{proof}
  We argue the first two claims of Item (1) together.
  Namely, we show that $R\in\mathcal Q$ if and only if
  $R$ belongs to the quasivariety of commutative rings
  satisfying $(xyz=z)\to (yz=z)$.

For the ``if'' part, let $R$ be a commutative ring satisfying 
$(xyz=z)\to (yz=z)$. Choose $r\in R$ and assume
that $(r)=(s)$ for some $s$. Then $s=qr$ and $r=ps$
for some $p, q\in R$. Since $pqr=r$, the quasi-identity
yields $qr=r$, or $s=r$. Thus, $(r)=(s)$ implies $r=s$,
showing that $R$ satisfies the unique generator
property for principal ideals.
Conversely, for ``only if'', suppose that
$R$ does not satisfy $(xyz=z)\to (yz=z)$.
$R$ must have elements $p, q, r$ such that $pqr=r$ and $qr\neq r$,
Then $(r)=(qr)$ and $qr\neq r$, so $R$ does not have the unique generator
property.

For the second to last statement of Item (1), suppose that
$R\in {\mathcal Q}$ and that $u$ is a unit in $R$.
Then $(u)= R = (1)$, so by the unique
generator property $u=1$. Also, to see that $R$ is reduced,
assume that $n\in R$ satisfies $n^2=0$. Then
$1+n$ is a unit (with inverse $1-n$), so $1+n=1$,
so $n=0$.

For the final statement of Item (1), the fact that any 
$R\in{\mathcal Q}$ is an $\mathbb F_2$-algebra follows from the fact
that $-1$ is a unit, so $1=-1$. Then the prime subring
of $R$ is isomorphic to $\mathbb F_2$, which is enough
to establish that $R$ is an $\mathbb F_2$-algebra.

For Item (2), if 
$D\in {\mathcal D}$, then $D$ is a domain 
by definition, and it has trivial unit group by Item (1).
Conversely, suppose that $D$ is a domain with trivial unit
group. If $(a)=(b)$ in $D$, then $a$ and $b$ must differ
by a unit, hence $a=b$, showing that $D$ has the 
unique generator property, so $D$ is a domain
in $\mathcal Q$, yielding $D\in\mathcal D$.

In order to establish Item (3) we first prove a claim.

\begin{clm}\label{basic_claim}
If $R\in {\mathcal Q}$ and $S\subseteq R$ is a subset, 
then the annihilator $A=\ann(S)$ is a 
$\mathcal Q$-ideal (meaning that 
$R/A\in {\mathcal Q}$).
\end{clm}

{\it Proof of claim.}
For this we
must verify that $R/A$ satisfies the quasi-identity
$(xyz=z)\to (yz=z)$. Equivalently, we must show 
that if $x, y, z\in R$ and
$xyz\equiv z\pmod{A}$, then
$yz\equiv z\pmod{A}$. We begin:
If $xyz\equiv z\pmod{A}$, then
$(xyz-z)\in A$, so $(xyz-z)s =0$
for any $s\in S$.
This means that $xy(zs)=(zs)$ for any $s\in S$.
Applying the quasi-identity from
Item (1) with $zs$ in place of $z$ we 
derive that $yzs=zs$, or $(yz-z)s=0$ for any $r\in I$.
Hence $yz\equiv z\pmod{A}$, as desired.

\medskip

Next we argue that if $R\in{\mathcal Q}$ is not a domain, then
$R$ has disjoint nonzero $\mathcal Q$-ideals $I$ and $J$.
If $R$ is not a domain, then there exist nonzero $r$
and $s$ such that $rs=0$. Take $I=\ann(r)$ 
and $J = \ann(I)$. $I$ is nonzero since it contains
$s$, and $J$ is nonzero since it contains $r$.
Both $I$ and $J$ are $\mathcal Q$-ideals
by Claim~\ref{basic_claim}.
If $t\in I\cap J$, then $t^2\in IJ = \{0\}$,
so $t$ is nilpotent. According to Item (1),
any $R\in{\mathcal Q}$ is reduced, so $t=0$. Thus 
$I$ and $J$ are indeed disjoint nonzero 
$\mathcal Q$-ideals.

The argument for Item (3) is completed by noting that 
any quasi-variety $\mathcal Q$ is expressible as
$\Sub\Prod(\mathcal K)$ where $\mathcal K$
is the subclass of relatively subdirectly irreducible
members of $\mathcal Q$.
This is a version of Birkhoff's Subdirect Representation
Theorem, stated for quasivarieties, and it holds for quasivarieties
because relative ideal/congruence lattices are algebraic.
The previous paragraph shows that the only members
of ${\mathcal Q}$ that could possibly be relatively 
subdirectly irreducible are the domains.
(That is, 
$R$ not a domain $\Rightarrow$ 
$R$ has disjoint nonzero $\mathcal Q$-ideals
$\Rightarrow$ 
$R$ is not relatively subdirectly irreducible.)

To prove Item (4), we refer to general criteria from
\cite{kearnes-mckenzie} for proving that a quasivariety
is relatively congruence distributive. Specifically
we will use Theorems 4.1 and 4.3 of that paper,
along with some of the remarks between those theorems.

Here is a summary of what we are citing.
From Theorem 4.1 of \cite{kearnes-mckenzie},
a quasivariety is relatively congruence modular
if and only if it satisfies the ``extension principle''
and the ``relative shifting lemma''.
From remarks following the proof of Theorem 4.1,
the ``extension principle'' can be replaced by
the ``weak extension principle''. From Theorem 2.1 of that paper,
the ``relative shifting lemma'' can be replaced by the
``existence of quasi-Day terms''.
Finally, from Theorem 4.3 of that paper, a quasivariety
is relatively congruence distributive if and only if
it is relatively congruence distributive and
no member has a nonzero abelian congruence.

What this reduces to in our setting is this:
to prove that our quasivariety $\mathcal Q$ is relatively
ideal distributive (Item (4)) it suffices to show that $\mathcal Q$ 
(i) has ``quasi-Day terms'', 
(ii) satisfies the ``weak extension principle'', and 
(iii) has no member with a nontrivial
abelian congruence (i.e., a nonzero ideal $A$
satisfying $A^2=0$).

Condition (i) holds since $\mathcal Q$
has ordinary Day terms, in fact a Maltsev term.
(More explicitly, the singleton set
$\Sigma_s:=\{(p(w,x,y,z),q(w,x,y,z)\}$
where $p(w,x,y,z):=w-x+y$ and $q(w,x,y,z):=z$
meets the defining conditions from  Theorem 2.1(2) 
of \cite{kearnes-mckenzie} for ``quasi-Day terms''.)

Condition (iii) holds since if $A\lhd R\in {\mathcal Q}$
and $A^2=0$, then the elements of $A$ are nilpotent.
As argued in the proof of Item (1),
the only nilpotent element in $R$ is $0$,
hence $A=0$.

Condition (ii) means that if $R\in {\mathcal Q}$
has disjoint ideals $I$ and $J$, then $I$ and $J$ can be extended
to $\mathcal Q$-ideals $\overline{I}\supseteq I$ and
$\overline{J}\supseteq J$ that are also disjoint. 
To prove that Condition (ii) holds we modify
an argument from above:
If $R$ has ideals $I$ and $J$ such that $I\cap J=0$,
then $IJ=0$. 
The ${\mathcal Q}$-ideal $\overline{J}=\ann(I)$ contains $J$,
the ${\mathcal Q}$-ideal $\overline{I}=\ann(\overline{J})$ 
contains $I$, both
are $\mathcal Q$-congruences, and $\overline{I}\cap \overline{J} = 0$
(since the elements in this intersection 
square to zero and $R$ is reduced).
This shows that disjoint ideals $I$ and $J$
may be extended to disjoint $\mathcal Q$-ideals.

For Item (5), to show that a locally finite ring in $\mathcal Q$ 
is a Boolean ring it suffices to show that any finite
ring $F\in \mathcal Q$ is Boolean. (The reason this
reduction is permitted is that
the property of being a Boolean ring is expressible
by the identity $x^2=x$, and a locally finite
structure satisfies a universal sentence
if and only if its finite substructures satisfy the sentence.)

So choose a finite $F\in \mathcal Q$.
As $F$ has trivial unit group, and 
$1+{\rm rad}(F)\subseteq U(F)$, we get that $F$ must be semiprimitive.
Since $F$ is finite it must be a product of fields.
Since $F$ has only trivial units, each factor field 
must have size $2$, so $F$ is Boolean.

Conversely, if $B$ is any Boolean ring,
then multiplication is a semilattice operation,
so $xyz\leq yz\leq z$ in the semilattice order
for any $x, y, z\in B$.
If, in $B$, we have first = last ($xyz=z$), then we must
have middle = last ($yz=z$). Hence $B\in{\mathcal Q}$.

To complete the proof of Item (5)
we must show that if $\mathcal V$ is a variety and
${\mathcal V}\subseteq {\mathcal Q}$, then $\mathcal V$
consists of Boolean rings. For this it suffices to
show that if $R\in {\mathcal Q}$ is not Boolean
(i.e., $R$ has an element $r$
satisfying $r\neq r^2$), then $R\notin {\mathcal V}$.
This holds because $\langle r^2\rangle\subsetneq \langle r\rangle$
by the unique generator property, so 
$r/\langle r^2\rangle$ is a nonzero nilpotent element of
$R/\langle r^2\rangle$, establishing that some homomorphic
image of $R$ is not in $\mathcal Q$.
\end{proof}

By substituting $z=1$ in the quasi-identity
$(xyz=z)\to (yz=z)$ we obtain the consequence $(xy=1)\to (y=1)$,
which expresses that the unit group is trivial.
Since a consequence can be no stronger than the original statement,
this is enough to deduce that the quasivariety
of commutative rings with trivial unit group
contains the quasivariety of commutative rings
whose principal ideals have unique generators.
This containment is proper, and the following example
describes a commutative ring satisfying
$(xy=1)\to (y=1)$ but not
$(xyz=z)\to (yz=z)$.

\begin{example}\label{triv}
  Let $R$ be the commutative $\mathbb F_2$-algebra presented by
  \[
  \langle X, Y, Z\mid XYZ=Z\rangle.\]
  That is,
  $R$ is the quotient of the polynomial ring
  $\mathbb F_2[X, Y, Z]$ by the ideal $(XYZ-Z)$.
  
  We may view the relation $XYZ-Z=0$ as a reduction rule
  $XYZ\to Z$ to produce a normal form for elements of $R$.
  This single rule is applied as follows: choose 
  a monomial of the form $XYZW$ ($W$ is a
  product of variables)
  of an element in a coset of $(XYZ-Z)\subseteq \mathbb F_2[X,Y,Z]$
  and replace $XYZW$ by $ZW$ 
  That is, if each of $X, Y, Z$ appear in
  a monomial, we delete one instance of $X$ and one instance of $Y$
  from that monomial.

  The Diamond Lemma applies to show that there is a
  normal form for elements of $R$, and the elements in normal
  form are exactly the polynomials over $\mathbb F_2$
  in the generators $X, Y, Z$ where no monomial is divisible
  by each of $X, Y$, and $Z$.

  Note that each application of the reduction rule reduces the
  $X$-degree and the $Y$-degree of some monomial, but does
  not alter the $Z$-degree of any monomial. This is enough to prove
  that the unit group of $R$ is trivial. For if $R$ had a unit
  $u$ with inverse $v$, then the $Z$-degree of the product $uv=1$
  is zero, but it is also
  the sum of the $Z$-degrees of $u$ and $v$.
  Hence the normal form of a unit must be $Z$-free. But then $u$
  and $v$ would then be inverse units in the subring
  $\mathbb F_2[X,Y]$, where all elements are in normal form.
  Now one can argue in this subring, using $X$-degree and $Y$-degree,
  to conclude that none of $X, Y, Z$ appear in the normal form
  of a unit. We are left with $u=v=1$ as the only possibility.

  Notice also that $YZ-Z$ is in normal form, so $YZ-Z\neq 0$
  in $R$. This shows that $R$ fails to satisfy
  $(xyz=z)\to (yz=z)$,
  but does satisfy $(xy=1)\to (y=1)$.
  In particular, the fact that $XYZ=Z$ while $YZ\neq Z$
  means that $(YZ)=(Z)$, while $YZ\neq Z$,
  so the principal ideal $(Z)$ does not have a unique generator.
\end{example}

\section{Some related quasivarieties}
\label{s:3}
Recall from the introduction that one of our goals
was to identify the situation when the passage
from the monoid $\langle R; \cdot\rangle$ to the poset of principal
ideals results in no loss of information. This led
to the investigation of rings whose principal ideals
have unique generators. For these rings we progressed from
(1) to (2) to (3) on this list of equivalent conditions:
\begin{enumerate}
  \item[(1)] Principal ideals of $R$ have unique generators.
  \item[(2)] $R$ satisfies $(xyz=z)\to (yz=z)$.
  \item[(3)] $R$ belongs to the quasivariety generated by
    domains with trivial unit group, which is the class
    of domains satisfying $(xy=1)\to (y=1)$.
\end{enumerate}
If we had progressed in the reverse direction, taking
the property of ``trivial unit group''
as the object of interest, the results would have been different.
The quasivariety of commutative rings with trivial unit
group is not generated by its subclass of domains, and it
is not relatively ideal distributive.

Notice also that progressing in the reverse direction
and trying to obtain the same results 
requires one to strengthen $(xy=1)\to (y=1)$ to
$(xyz=z)\to (yz=z)$.

In this section, we \emph{will} progress in the reverse direction
starting with a generalization of the quasi-identity
$(xy=1)\to (y=1)$. Namely, we will investigate
the quasivariety of commutative rings generated by 
domains satisfying $(xy=1)\to (y^n=1)$, which asserts
that every unit has exponent dividing $n$.
(The previous section was entirely concerned with the case $n=1$.)

If $D$ is a domain satisfying $(xy=1)\to (y^n=1)$, then
its group of units is a cyclic group of order dividing $n$.
The reason for this is that unit group $U(D)$
is an abelian group satisfying $x^n=1$, hence $U(D)$ is
a locally finite abelian group. If $U(D)$ is not cyclic, then
it contains a finite noncyclic subgroup $G\subseteq U(D)$.
But now $G$ is a finite noncyclic subgroup of the field of fractions
of $D$, and we all know that the multiplicative group
of a field contains no finite noncyclic subgroup.

So let ${\mathcal Q}_{|n}$ denote the quasivariety
of commutative rings generated by the domains
$D$ for which $U(D)$ is cyclic of order dividing $n$.
We shall show that this quasivariety is relatively ideal distributive
and we shall provide an axiomatization for 
${\mathcal Q}_{|n}$.

Write ${\mathcal D}_{|n}$ for the class of domains
whose unit group is cyclic of order dividing $n$.

\begin{theorem}
  \label{ak2}
  By definition, we have that ${\mathcal Q}_{|n}$ is the quasivariety
  generated by ${\mathcal D}_{|n}$.
\begin{enumerate}
\item[(1)] 
  ${\mathcal Q}_{|n}$ is axiomatized
  by
\begin{enumerate}
\item[(a)] the identities defining commutative rings,
\item[(b)] the quasi-identity $(x^2=0)\to (x=0)$,
  which expresses that the only nilpotent element is $0$, and
\item[(c)] the quasi-identity 
  $(xyz=z)\to(y^nz=z)$.
\end{enumerate}
\item[(2)]
${\mathcal Q}_{|n}$ is a relatively ideal distributive quasivariety.
\end{enumerate}
\end{theorem}

\begin{proof}
  To prove Item (1),
  let $\mathcal K$ be the quasivariety axiomatized by
  the sentences in (a), (b), and (c).
  It is easy to see that 
  ${\mathcal D}_{|n}$ satisfies the quasi-identities in (a), (b) and (c), so
  ${\mathcal D}_{|n}\subseteq {\mathcal K}$, and therefore 
  ${\mathcal Q}_{|n}\subseteq {\mathcal K}$.

  Conversely, we must show that ${\mathcal K}\subseteq {\mathcal Q}_{|n}$
  For this, we need the analogue of Claim~\ref{basic_claim}
  for ${\mathcal K}$:

\begin{clm}\label{basic_claim_2}  
  If $R\in {\mathcal K}$
  and $S\subseteq R$ is a subset, 
then the annihilator $A=\ann(S)$ is a 
${\mathcal K}$-ideal.
\end{clm}

{\it Proof of claim.}
Our goal is to prove that $R/A\in{\mathcal K}$,
so we must prove that $R/A$ is a commutative
ring satisfying $(x^2=0)\to (x=0)$ and 
$(xyz=z)\to (y^nz=z)$.
It is clear that $R/A$ is a commutative ring
(identities are preserved under quotients),
so we only need to verify that $R/A$ satisfies 
$(x^2=0)\to (x=0)$ and 
$(xyz=z)\to (y^nz=z)$.
For the second of these, the proof
is exactly like the proof of Claim~\ref{basic_claim},
while for the first there is an extra idea.
We prove the first only.

To prove that $R/A$ satisfies $(x^2=0)\to (x=0)$,
we must show that $R$ satisfies
$x^2\equiv 0\pmod{A}$ implies 
$x\equiv 0\pmod{A}$.
If $x^2\equiv 0\pmod{A}$, or $x^2\in A$, then
$x^2s=0$ for all $s\in S$.
This implies $(xs)^2=(x^2s)s=0$
for all $s\in S$. (This is the ``extra idea''.)
But $R$ satisfies $(x^2=0)\to (x=0)$,
so from $(xs)^2=0$ we deduce $xs=0$
for all $s\in S$. This proves that $x\in A$
or $x\equiv 0\pmod{A}$.

\medskip

We will use Claim~\ref{basic_claim_2}
the same way we used
Claim~\ref{basic_claim} in the proof of Theorem~\ref{ak1}.
If $R\in\mathcal K$ is not a domain, then
there exist nonzero $r$
and $s$ such that $rs=0$.
Take $I=\ann(r)$ 
and $J = \ann(I)$. $I$ is nonzero since it contains
$s$, and $J$ is nonzero since it contains $r$.
By Claim~\ref{basic_claim_2}, 
$I$ and $J$ are $\mathcal K$-ideals.
Any element in $I\cap J$ must square to zero,
so since $\mathcal K$ satisfies axiom (b)
we get $I\cap J=\{0\}$.
Thus, if $R$ is not a domain, then it has
a pair of nonzero, disjoint, ${\mathcal K}$-ideals.
This is enough to guarantee that $R$ is not
subdirectly irreducible relative to $\mathcal K$.

In the contrapositive form, we have shown that
any relatively subdirectly irreducible member
of $\mathcal K$ is a domain. Hence
${\mathcal K}$ is generated by its subclass
of domains.

But if $D\in{\mathcal K}$ is a domain, then
by (1)(c) we derive that $D$ satisfies 
$(xyz=z)\to (y^nz=z)$. For domains we can cancel $z$
to obtain that $D$ satisfies $(xy=1)\to (y^n=1)$.
This shows that the domains in ${\mathcal K}$
lie in ${\mathcal D}_{|n}$, so ${\mathcal K}$
is contained the quasivariety generated by
${\mathcal D}_{|n}$, which is ${\mathcal Q}_{|n}$.

Item (2) of this theorem is proved exactly like
Item (4) of Theorem~\ref{ak1}.
\end{proof}

\begin{observation}\label{obs}
A quick test to rule out that some nonzero ring $R$ belongs to some
  quasivariety $\mathcal Q_{|n}$ is to show that
  the prime subring of $R$ does not belong to $\mathcal Q_{|n}$.
  Since the prime subring of $R$
  is isomorphic to $\mathbb Z_k$ for some $k>1$,
  and since the units of $\mathbb Z_k$ are
  easy to calculate, it is not hard to derive some consequences.

  In fact, it is not hard to write down the exact relationship
  between $k$ and $n$ which guarantees that
  $\mathbb Z_k\notin {\mathcal Q}_{|n}$, but here we only mention that
  if $n$ is odd, then $\mathbb Z_k\notin {\mathcal Q}_{|n}$
  unless $k=2$. Also, $\mathbb Z\notin {\mathcal Q}_{|n}$.

  To see this, note that if $R=\mathbb Z_k$ for $k\neq 2$
  or $R=\mathbb Z$, then $1\neq -1$, 
  so $U(R)$ contains an element $-1$ of order $2$.
  But if $n$ is odd, then the rings in ${\mathbb Q}_{|n}$
  have unit groups of odd exponent, hence do not have units
  of order $2$.

  In particular, this shows that if $n$ is odd, then
  each ring in ${\mathbb Q}_{|n}$ may be thought of
  as a $\mathbb F_2$-algebra.
\end{observation}

We can use this observation to show that not all the quasivarieties
${\mathcal Q}_{|n}$ are distinct, in particular

\begin{theorem}
  If $p$ is an odd prime, then ${\mathcal Q}_{|p}={\mathcal Q}_{|1}$
  unless $p$ is a Mersenne prime.
\end{theorem}

\begin{proof}
  To prove that ${\mathcal Q}_{|p}={\mathcal Q}_{|1}$
  when $p$ is an odd non-Mersenne prime, it will suffice
  to show that these quasivarieties contain the same domains.
  We always have ${\mathcal Q}_{|m}\subseteq {\mathcal Q}_{|n}$
  when $m\mid n$,
  from the definition of these quasivarieties, so we must show
  that any domain $D\in {\mathcal Q}_{|p}$ is contained in
  ${\mathcal Q}_{|1}$ (i.e.  has a trivial unit group).

  Choose $D\in {\mathcal Q}_{|p}$.
  From Observation~\ref{obs}, we know (since $p$ is odd)
  that $D$ is an $\mathbb F_2$-algebra.
  Suppose that $\theta\in D$ is a nontrivial unit.
  Since $\theta$ has finite multiplicative order,
  and the prime subring of $D$
  is finite, the subring $S\subseteq D$ generated by $\theta$
  is finite. $S$ is a subring of a domain itself,
  hence it is a field, and $U(S)=S^{\times}$.
  $S$ belongs to ${\mathcal Q}_{|p}$,
  so $S^{\times}$ has order dividing $p$, and it must therefore be that
  $|S^{\times}|=p$. This shows that $S$ is a finite field
  of characteristic $2$ and of cardinality $|S|=p+1$.
  We derive that $p+1=2^s$ for some $s$, or
  $p=2^s-1$. This completes the proof that
  ${\mathcal Q}_{|p}={\mathcal Q}_{|1}$ unless
  $p$ is a Mersenne prime.  
\end{proof}

The primality of $p$ did not play a big role in the proof.
The same argument shows that if $n$ is any odd number,
then ${\mathcal Q}_{|n}={\mathcal Q}_{|1}$ unless $n$
is divisible by some number $x>1$ of the form
$x=2^s-1$. So, for example,
${\mathcal Q}_{|55}={\mathcal Q}_{|25}={\mathcal Q}_{|1}$.
But if $n$ is divisible by some number $x>1$ of the form
$x=2^s-1$, then ${\mathcal Q}_{|n}$
will contain some finite fields that are not in
${\mathcal Q}_{|1}$.

\section{A representation theorem for Vasconcelos monoids}
\label{s:5}

In this section we will focus on the structure ${\mathcal I}_{\textrm{f.g.}}(R)$
formed by the finitely generated ideals of a commutative unital ring $R$.
We will refer to ${\mathcal I}_{\textrm{f.g.}}(R)$ as \emph{the divisibility
theory of $R$}. Traditionally,
${\mathcal I}_{\textrm{f.g.}}(R)$ is considered partially ordered under
reverse inclusion: $I\le_l J$ $\Leftrightarrow$ $I\supseteq J$.
Thus $\langle{\mathcal I}_{\textrm{f.g.}}(R);\wedge\rangle$
is a meet subsemilattice of the dual ${\mathcal I}^d(R)$
of the ideal lattice of $R$ where
$I\wedge J=I+J$ for all $I,J\in{\mathcal I}_{\textrm{f.g.}}(R)$.
In general, ${\mathcal I}_{\textrm{f.g.}}(R)$ is \emph{not} a sublattice of
${\mathcal I}^d(R)$, because the intersection of finitely generated ideals of
$R$ need not be finitely generated.

Ideal multiplication induces another monoid structure
$\langle{\mathcal I}_{\textrm{f.g.}}(R);\cdot\rangle$
on ${\mathcal I}_{\textrm{f.g.}}(R)$.
Divisibility with respect to this multiplication
yields another partial order on ${\mathcal I}_{\textrm{f.g.}}(R)$:
for any $I,J\in{\mathcal I}_{\textrm{f.g.}}(R)$, we let
$I\le_m J$ $\Leftrightarrow$ $J=IK$ for  some
$K\in {\mathcal I}_{\textrm{f.g.}}(R)$.
In general, the partial orders $\le_l$ and $\le_m$ are different.

However, if $R$ is \emph{arithmetical}~\cite{fu},
that is, if the ideal lattice of $R$ is
distributive, then
\begin{itemize}
\item
  the partial orders $\le_l$ and $\le_m$ coincide
  (see \cite[Theorem~2]{j2} and \cite[p.~3972]{amv}),
  and will be denoted by $\le$; furthermore,
\item
  the intersection of any two finitely generated ideals is finitely
  generated~\cite[Proposition~1.7]{amv}.
\end{itemize}
Consequently,  
\begin{itemize}
\item
  $\langle{\mathcal I}_{\textrm{f.g.}}(R);\wedge,\vee\rangle$ is a sublattice
  of ${\mathcal I}^d(R)$; namely,
  $I\wedge J=I+J$ and $I\vee J=I\cap J=I_1J$
  for all $I,J\in{\mathcal I}_{\textrm{f.g.}}(R)$, where
  $I_1\in{\mathcal I}_{\textrm{f.g.}}(R)$ satisfies $I=(I\wedge J)I_1$
  (see \cite[Proposition~1.2 and its proof]{amv}).
\end{itemize}
In fact, it follows that
\begin{itemize}
\item
  $\langle{\mathcal I}_{\textrm{f.g.}}(R);\cdot,\le,\wedge,\vee\rangle$
  is a B\'ezout monoid as defined below (see \cite[Proposition~1.6]{amv}).
\end{itemize}

\begin{definition} (See \cite[Section~4]{B3}, \cite[Definition~1.1]{amv}.)
\label{basic}
A \emph{B\'{e}zout monoid} $S$ is a commutative monoid $\langle S;\cdot\rangle$
with $0$ such that
\begin{enumerate}
\item[(1)]
  the divisibility relation $a\mid b\Leftrightarrow aS\supseteq bS$
  is a partial order $\le $ on $S$, called the
  \emph{natural partial order of $S$},
\item[(2)]
  $\le$ is a lattice order, and the induced lattice
  $\langle S;\wedge,\vee\rangle$ is distributive;
\item[(3)]
  $\cdot$ is distributive over $\wedge$ and $\vee$; and
\item[(4)]
  $S$ is \emph{hyper-normal}, that is,
  for any elements $a,b\in S$, if $d=a\wedge b$ and $a_1\in S$ is such that 
  $a=da_1$, then there exists an element $b_1\in S$ satisfying $a_1\w b_1=1$ and
  $b=db_1$.
\end{enumerate}
\end{definition}

As we mentioned right before this definition, the divisibility theory
${\mathcal I}_{\textrm{f.g.}}(R)$ of an arithmetical ring $R$ is a B\'ezout monoid.
An major open problem is whether the converse of this statement is also true;
that is: Is every B\'ezout monoid isomorphic to the divisibility theory of an
arithmetical ring?
Primary examples of arithmetical rings are B\'ezout rings, that is,
commutative unital rings in which all finitely generated ideals are principal.
However, there are many other classes of arithmetical rings, e.g.,
semihereditary rings and rings of weak dimension $1$.
Nevertheless, by a result of Anderson~\cite{and}, the divisibility
theory ${\mathcal I}_{\textrm{f.g.}}(R)$
of every arithmetical ring $R$ is isomorphic to the divisibility theory
${\mathcal I}_{\textrm{f.g.}}(\sf{R})={\mathcal I}_{\textrm{princ}}(\sf{R})$
of a B\'ezout ring $\sf{R}$.
Therefore, the question above is equivalent to the following:
\begin{center}
  {\it
Is every B\'ezout monoid isomorphic to the divisibility theory of a B\'ezout
ring?}
\end{center}
The answer is known to be YES for B\'ezout monoids that are either
semihereditary~\cite{as1} or have only one minimal $m$-prime filter~\cite{as2}
(for the definitions, see the forthcoming paragraphs).

Our aim in this section is to prove a similar representation theorem
for a class of B\'ezout monoids that we will call \emph{Vasconcelos monoids}.
The name is motivated by the fact that a well-known example of a
B\'ezout ring constructed by
Vasconcelos~\cite[Example~3.2]{vas} (see also \cite[Example~3]{ma})
has its divisibility theory in this class.

Before defining Vasconcelos monoids we recall some definitions and basic facts
on B\'ezout monoids.
Let $S$ be a B\'ezout monoid.
For any element $a\in S$ we let $a^\perp$ denote the annihilator
$\{x\in S:ax=0\}$ of $a$.
By a \emph{filter of $S$} we mean a lattice
filter of $S$, that is, a subset $F$ of $S$ such that $F$ is up-closed
(i.e., $b\in F$ whenever $S\ni b\ge a\in F$) and is closed under $\wedge$.
Clearly, every filter $F$ of $S$ is a monoid ideal of $S$ as well, that is,
$F$ satisfies $FS\subseteq F$. Furthermore, properties~(2) and (3)
in the definition of a B\'ezout monoid imply that $a^\perp$ is
a filter of $S$ for every element $a\in S$.
A filter $F$ of $S$ is said to be
\emph{multiplication prime} (briefly, \emph{$m$-prime}) if $F\not=S$ and
$ab\in F$ implies $a\in F$ or $b\in F$ for all $a,b\in S$.
If $F$ is a filter of $S$, we define $S/F$ to be the factor
of $S=\langle S;\cdot,\le,\wedge,\vee\rangle$ by the congruence $\equiv_F$
defined for all $a,b\in S$ 
by $a\equiv_F b$ iff $a\wedge c=b\wedge c$ for some $c\in F$.
It is not hard to show (see~\cite[Theorem~2.15]{amv})
that $S/F$ is a B\'ezout monoid. 
The following two basic properties of B\'ezout monoids will also
be important later on.

\begin{proposition}
  \label{prop-bezout1}
  If $S$ is  B\'ezout monoid, then
  \begin{enumerate}
  \item[{\rm(1)}]
    $1$ is the unique invertible element of $S$; and
  \item[{\rm(2)}]
    the idempotent elements of $S$ form a Boolean
    algebra, and for any two idempotent elements
    $e,f$ of $S$ we have $e\vee f=ef$.~\cite[Propositions~1.14, 1.16]{amv}
  \end{enumerate}
\end{proposition}  

It follows from statement (2) above that
if $S$ is a B\'ezout monoid, then
for any idempotent $e^2=e\in S$ there exists a unique idempotent $f^2=f\in S$
such that $e\wedge f=1$ and $e\vee f=ef=0$. We will refer to $f$ as the
\emph{orthogonal complement} of $e$.

For a B\'ezout monoid $S$, we say that
\begin{itemize}
\item
  $S$ is \emph{semihereditary} if for each $a\in S$ there exists an idempotent
  $e^2=e\in S$ such that $a^\perp=eS$; and
\item
  $S$ is \emph{semiprime} if it has no non-zero nilpotent elements, or
  equivalently, if the intersection of its $m$-prime filters is $\{0\}$.
\end{itemize}
It is known that semihereditary B\'ezout monoids are
semiprime~\cite[Corollary~2.2]{as1}.

For the rest of this section we will focus on B\'ezout monoids $S$ that are
semiprime.
For each $a\in S$ let $S_a=\{b\in S:b^\perp=a^\perp\}$.
In particular, $S_1$ is the set of all elements $b\in S$ satisfying
$b^\perp=1^\perp=\{0\}$;
that is, $S_1$ is the set of all elements of $S$
that are not zero divisors (in $S$).

\begin{proposition} 
  \label{prop-bezout2}
        \cite[Proposition~3.3, Theorem~3.5]{amv} 
  Let $S$ be a semiprime B\'ezout monoid, and let $B$ be the Boolean
  algebra of all idempotent elements of $S$.
  \begin{enumerate}
  \item[{\rm(1)}]
    $(st)^\perp=(s\vee t)^\perp$ and
    $(s\wedge t)^\perp=s^\perp\cap t^\perp$
    hold for arbitrary elements $s,t\in S$.
  \item[{\rm(2)}]
    For every element $a\in S$,
    \begin{enumerate}
    \item[{\rm(i)}]
    $S_a$ is a cancellative subsemigroup of $S$, which is closed under
    $\wedge$ and $\vee$, and
    \item[{\rm(ii)}]
    the map $S_a\to S/a^\perp$, $x\mapsto x/{\equiv_{a^\perp}}$ embeds $S_a$
    into the subsemigroup $(S/a^\perp)_1$ of $S/a^\perp$; moreover,
    $S_a$ and $(S/a^\perp)_1$ have the same group of quotients.
    \end{enumerate}
  \item[{\rm(3)}]
    $S$ is the disjoint union of its distinct
    (cancellative) submonoids $S_a$, where $S_0=\{0\}$
    and $S_1$ consists of all elements of $S$ that are
    not zero divisors in $S$.
    In particular, $S_{e_1}\cap S_{e_2}=\emptyset$ whenever
    $e_1,e_2$ are different elements in $B$.
  \item[{\rm(4)}]  
    $S_e=eS_1$ holds for every $e\in B$.
  \end{enumerate}
\end{proposition}  

\begin{proof}
  For a proof of (1), (2), and the first statement in (3) we refer
  the reader to \cite[Proposition~3.3]{amv} and \cite[Theorem~3.5]{amv}.
  For any $a,b\in S$ the
  definitions of $S_a$ and $S_b$ imply that
  $S_a\not=S_b\ \Leftrightarrow\ a^\perp\not=b^\perp\ \Leftrightarrow\ S_a\cap S_b=\emptyset$.
  Therefore the second statement in (3)
  is an immediate consequence of the fact that for $e_1,e_2\in B$
  we have $e_1^\perp\cap B\not=e_2^\perp\cap B$ whenever $e_1\not=e_2$.

  In (4) the inclusion $S_e\supseteq eS_1$ holds, because
  $(ec)^\perp=e^\perp$ for all $c\in S_1$.
  To prove the reverse inclusion $S_e\subseteq eS_1$, let $s\in S_e$, and let
  $f$ be the orthogonal complement of $e$.
  Then $fs=0$, because $s\in S_e$ implies that $f\in e^\perp=s^\perp$.
  Hence,
  \[
  s=1s=(e\wedge f)s=es\wedge fs=es\wedge 0=es\wedge ef
  =e(s\wedge f),
  \]
  where $s\wedge f\in S_1$, because
  $(s\wedge f)^\perp=s^\perp\cap f^\perp=e^\perp\cap f^\perp=(e\wedge f)^\perp
  =1^\perp=\{0\}$.
\end{proof}

Now we are ready to define Vasconcelos monoids.

\begin{definition}
  \label{vas1}
  A \emph{Vasconcelos monoid} is a B\'ezout monoid $S$ that satisfies the
  following conditions:
  \begin{enumerate}
  \item[(1)]
    $S$ is semiprime,
  \item[(2)]
    $S$ contains a zero divisor $a\not=0$ such that
    $P=a^{\perp}$ is an $m$-prime filter of $S$ which consists of idempotents,
    and
  \item[(3)]
    for the submonoid $C=S_1\cup S_a$ of $S$ (where $S_1S_a\subseteq S_a$),
    $C\cup\{0\}$ is isomorphic to $S/P=S/a^\perp$ via the map
    $x\mapsto x/{\equiv_P}$ ($x\in C\cup\{0\}$), hence $C\cup\{0\}$ is a
    B\'ezout monoid, and $C$ is isomorphic to the positive cone
    of the lattice ordered abelian group of quotients of $S_a$.
  \end{enumerate}
\end{definition}

We will need several consequences of this definition, which we
state and prove in the next proposition.

\begin{proposition}
  \label{prop-vas}
Let $S$ be a Vasconcelos monoid, let $a\in S$ be as in
Definition~\ref{vas1} above, and let $P=a^\perp$.
Furthermore, let $B$ denote the Boolean algebra
of all idempotent elements of $S$.
\begin{enumerate}
\item[{\rm(1)}]
$B\setminus P$ is the set of orthogonal complements
of the elements of $P$.  
\item[{\rm(2)}]
$ce=e$ holds for all $c\in S_1$ and $e\in P$; hence, 
  $P\subseteq cS$ for all $c\in S_1$.  
\item[{\rm(3)}]
  $de=0$ and $df=d$ hold for all $d\in S_a\cup\{0\}$,
  $e\in P$, and $f\in B\setminus P$.  
\item[{\rm(4)}]
$S$ is the disjoint union of its subsets
\begin{align*}  
S_1(B\setminus P) & {}=\{cf:c\in S_1,\ f\in B\setminus P\}\quad \text{and}\\
(S_a\cup\{0\})\wedge P & {}=\{d\wedge e: d\in S_a,\ e\in P\}.
\end{align*}
Moreover, every element of
$S_1(B\setminus P)$ can be written uniquely
as a product $cf$ with $c\in S_1$, $f\in B\setminus P$,
and every element of
$(S_a\cup\{0\})\wedge P$ can be written uniquely as a meet $d\wedge e$ with
$d\in S_a\cup\{0\}$, $e\in P$.
\item[{\rm(5)}]
If the elements of $S$ are written in the form described in (4), then
the monoid operation of $S$ takes the following form:
for all $c_1,c_2\in S_1$, $d_1,d_2\in S_a$, $e_1,e_2\in P$, and
$f_1,f_2\in B\setminus P$ where $f_i$ is the orthogonal complement of
$e_i$ for $i=1,2$,
\begin{align*}
  (c_1f_1)(c_2f_2)& {}=(c_1c_2)(f_1\vee f_2)
  \quad\text{where $c_1c_2\in S_1$ and $f_1\vee f_2\in B\setminus P$},\\
  \qquad\ \ 
  (d_1\wedge e_1)(d_2\wedge e_2)& {}=d_1d_2\wedge(e_1\vee e_2)
  \quad\text{where $d_1d_2\in S_a\cup\{0\}$ and $e_1\vee e_2\in P$},\\
  (c_1f_1)(d_2\wedge e_2) & {}=c_1d_2\wedge e
  \quad\text{where $c_1d_2\in S_a\cup\{0\}$ and $e=f_1e_2\in P$ is}\\
  &\qquad\qquad\qquad\text{the relative complement of $e_1e_2$ in $e_2$,}\\
  &\qquad\qquad\qquad\text{i.e., $e\in P$ with
    $e_1e_2\wedge e=e_2$, $e_1e_2\vee e=0$.}
\end{align*}  
\end{enumerate}
\end{proposition}

\begin{proof}
To prove (1), recall that  
$P\subseteq B$ is an $m$-prime filter of $S$.
Therefore,
if $e,f$ is any pair of orthogonal idempotents in $B$, then 
$ef=0\in P$ and $e\wedge f=1\notin P$ imply that exactly one of $e,f$
belongs to $P$, which proves statement (1).

For (2), let $c\in S_1$ and $e\in P$. Then $ce\in P$, since $P$ is a filter
of $S$, and $ce$ is idempotent, since $P\subseteq B$. But $c\in S_1$ implies
that $(ce)^\perp=e^\perp$, therefore it must be that $ce=e$.
Now the inclusion $P\subseteq cS$ is clear for all $c\in S_1$.

Statement (3) is trivial if $d=0$, so let us assume that $d\in S_a$.
In view of statement~(1), we may assume without loss of generality that
$f\in B\setminus P$ is the orthogonal complement of $e\in P$.
Since $d\in S_a$, we have $d^\perp=a^\perp=P$, so $de=0$.
Hence, $d=d1=d(e\wedge f)=de\wedge df=0\wedge df=df$.

To prove the first statement in (4),
let $s\in S$. By item~(3) in Definition~\ref{vas1},
there exists $c\in C\cup\{0\}=S_1\cup S_a\cup\{0\}$ such that
$s/{\equiv}_P=c/{\equiv}_P$, that is, $s\wedge e'=c\wedge e'$ for some
$e'\in P$. Let $f'$ denote the orthogonal complement of $e'$.
Then $e'f'=0$ and $e'\wedge f'=1$, hence
$sf'=sf'\wedge e'f'=(s\wedge e')f'=(c\wedge e')f'=cf'\wedge e'f'=cf'$.
Notice also that $se'\in P$, so $se'\in B$ (i.e., $se'$ is idempotent).
First, let $c\in S_1$. Then, by statement~(2), $ce'=e'$. Since
$se'\in B$, we get that $f=f'\wedge se'\in B$.
Clearly, $f\notin P$, because $f\le f'$ and $f'\notin P$.
Furthermore,
\[
cf=c(f'\wedge se')=cf'\wedge cse'=sf'\wedge se'=s(f'\wedge e')=s1=s, 
\]
where the third equality follows from the equalities $cf'=sf'$ and $ce'=e'$
proved earlier. This show that in the case $c\in S_1$, we have that
$s\in S_1(B\setminus P)$.
Now let $c\in S_a\cup\{0\}$. If $c=0$, that is, if $s\wedge e'=0\wedge e'=e'$
for some $e'\in P$, then $s\ge e'$ and $s\in P$. Hence $s\in \{0\}\wedge P$.
It remains to consider the case $c\in S_a$. Since $P$ is an $m$-prime filter
of $S$ and $s,c,f'\notin P$, we get that $sf'=cf'\notin P$.
Next we want to argue that $cf'\in S_a$, that is, $(cf')^\perp=a^\perp=P$.
The inclusion $\supseteq$ is clear, because $c\in S_a$. For $\subseteq$, let
$x\in(cf')^\perp$, that is, $cf'x=0$. Then $cf'x\in P$ and $cf'\notin P$
imply that $x\in P$, as required. Thus, using again the equality $sf'=cf'$
we get that
\[
s=s1=s(f'\wedge e')=sf'\wedge se'=cf'\wedge se'
\]
with
$cf'\in S_a$ and $se'\in P$, which proves that $s\in S_a\wedge P$ in this case.

Our argument above shows that
the set $S_1(B\setminus P)$ is the union of the equivalence classes
$c/{\equiv}_P\in S/P$ with $c\in S_1$, while the set $(S_a\cup\{0\})\wedge P$
is the union of the equivalence classes
$c/{\equiv}_P\in S/P$ with
$c\in S_a\cup\{0\}$. Therefore condition (3) in Definition~\ref{vas1}
ensures that the sets $S_1(B\setminus P)$ and $S_1(B\setminus P)$
are disjoint.
This completes the proof of the first statement in (4).

Now we prove the uniqueness statement in (4).
Let $c_1,c_2\in S_1$ and $f_1,f_2\in B\setminus P$ be such that
$c_1f_1=c_2f_2$. Then, by Proposition~\ref{prop-bezout2}(4),
$c_1f_1=c_2f_2$ is an element of
$f_iS_1=S_{f_i}$ for $i=1,2$. Therefore $S_{f_1}$ and $S_{f_2}$ are not disjoint,
so by Proposition~\ref{prop-bezout2}(3)
it must be that $f_1=f_2$.
To show that $c_1=c_2$, consider the orthogonal complement $e$ of
$f_1=f_2\in B\setminus P$, and recall from statements (1) and (2)
that $e\in P$ and $c_1e=e=c_2e$.
Thus, $c_i = c_i1=c_i(e\wedge f_i)=c_ie\wedge c_if_i=e\wedge c_if_i$ for
$i=1,2$, whence $c_1=e\wedge c_1f_1=e\wedge c_2f_2=c_2$. 
This proves the uniqueness
statement for the subset $S_1(B\setminus P)$ of $S$.
Now let $d_1,d_2\in S_a\cup\{0\}$ and $e_1,e_2\in P$ be such that
$d_1\wedge e_1=d_2\wedge e_2$.
Then all four products $d_1e_1, d_1e_2, d_2e_1, d_2e_2$ are $0$
by statement~(3).
It follows that $d_1^2=d_1d_2$, because
\[
d_1^2=d_1^2\wedge 0=d_1^2\wedge d_1e_1=d_1(d_1\wedge e_1)
=d_1(d_2\wedge e_2)=d_1d_2\wedge d_1e_2=d_1d_2\wedge 0=d_1d_2.
\]
Since $S_a$ is a cancellative submonoid of $S$
(by Proposition~\ref{prop-bezout2}(2)), this forces
$d_1=d_2$. A similar calculation shows that $e_1^2=e_1e_2$ holds as well:
\[
e_1^2=0\wedge e_1^2=d_1e_1\wedge e_1^2=(d_1\wedge e_1)e_1
=(d_2\wedge e_2)e_1=d_2e_1\wedge e_2e_1=0\wedge e_2e_1=e_1e_2.
\]
By switching the roles of $d_1\wedge e_1$ and $d_2\wedge e_2$ we also get that
$e_2^2=e_1e_2$.
Hence $e_1=e_1^2=e_1e_2=e_2^2=e_2$, which proves the uniqueness statement
for the subset $(S_a\cup\{0\})\wedge P$ of $S$.
The proof of statement (4) is now complete.

For the proof of statement (5) let $c_1,c_2,d_1,d_2,e_1,e_2,f_1,f_2$
satisfy the assumptions. For the elements $c_1,c_2,d_1,d_2$,
Proposition~\ref{prop-bezout2}(2)(i) implies
that $c_1c_2\in S_1$ and $d_1d_2\in S_a\cup\{0\}$.
The fact that $c_1d_2\in S_a\cup\{0\}$
is clear if $d_2=0$, and follows from $(c_1d_2)^\perp=d_2^\perp=a^\perp$,
since $c_1\in S_1$ and $d_2\in S_a$.
Statement (1) and the fact that $P$ is a filter clearly imply that
$e_1\vee e_2\in P$, $e=f_1e_2\in P$, and the idempotent $f_1\vee f_2$,
which is the orthogonal complement of $e_1\wedge e_2\in P$
belongs to $B\setminus P$.
Furthermore, it is straightforward to verify that $e=f_1e_2$
is the relative complement of $e_1e_2$ in $e_2$; indeed,
$e_1e_2\wedge e=e_1e_2\wedge f_1e_2=(e_1\wedge f_1)e_2=1e_2=e_2$
and $e_1e_2\vee e=e_1e_2\vee f_1e_2=(e_1\vee f_1)e_2=0e_2=0$.

Now we verify the three equalities in (5).
For the first one, clearly, $(c_1f_1)(c_2f_2)=(c_1c_2)(f_1f_2)$
where $f_1f_2=f_1\vee f_2$ by Proposition~\ref{prop-bezout1}(2).
For the second one, we have
$(d_1\wedge e_1)(d_2\wedge e_2)=d_1d_2\wedge d_1e_2\wedge d_2e_1\wedge e_1e_2$
where $d_1e_2=0=d_2e_1$ by statement (3). Hence,
$(d_1\wedge e_1)(d_2\wedge e_2)=d_1d_2\wedge e_1e_2$ where
$e_1e_2=e_1\vee e_2$ as before.
Finally, for the third equality,
$(c_1f_1)(d_2\wedge e_2)= c_1f_1d_2\wedge c_1f_1e_2$ where
$d_2f_1=d_2$ and $c_1e_2=e_2$ hold by statements (3) and (2), respectively.
Therefore, $(c_1f_1)(d_2\wedge e_2)= c_1d_2\wedge f_1e_2=c_1d_2\wedge e$,
completing the proof.
\end{proof}

The main result of this section is the following representation theorem.

\begin{theorem}
\label{vas2}
For every Vasconcelos monoid $S$ there exists a B\'ezout ring $R$
such that the divisibility theory of $R$ is isomorphic to $S$.
\end{theorem}

\begin{proof}
Let $S$ be a Vasconcelos monoid, and let $B$ denote the Boolean algebra
of all idempotent elements of $S$.
By Definition~\ref{vas1}, $S$ is semiprime,
and $S$ contains a zero divisor $a\not=0$ such that
conditions (2)--(3) of the definition hold for $P=a^\perp$ and
the submonoid $C=S_1\cup S_a$ of $S$.
Recall that $C\cup\{0\}$ is B\'ezout monoid isomorphic to $S/P$, and
$C$ is a cancellative monoid isomorphic to the positive cone
of the lattice ordered abelian group of quotients of $S_a$.

For the construction of the B\'ezout ring $R$ with divisibility theory $S$
we will use parts of the construction developed for the main theorem
of~\cite{as1}, which we will recall now.

Let $L$ be an arbitrary field, and let $H$ be
the localization of the semigroup algebra $LC$ at the set
\begin{equation*}
\biggl\{\sum_{i=1}^n k_is_i: 0\neq k_i\in L,\ s_i\in C,\ \wi s_i=1,\,\,
n\in \N\biggr\},
\end{equation*}
which is a B\'ezout domain
with divisibility theory isomorphic to $C\cup 0$.
Since $S$ is a Vasconcelos monoid, condition~(3) in Definition~\ref{vas1}
implies that the ideal $Q$ of $H$ generated by all elements $s\in S_a$
is a prime ideal, and hence $H/Q$ is an integral domain. Let $K$ denote
the field of fractions of $H/Q$, and note that
$K$ is isomorphic to the field of fractions of the semigroup algebra $LS_1$.

The second piece of the construction we need from~\cite{as1} is the
ring $C_K(\X)$ of all continuous functions from the spectrum $\X$
of the Boolean algebra $B$ of all idempotents of $S$ endowed with
the Zariski topology, into the field $K$ endowed with
the discrete topology.
For each $e\in B$ let $D_e$ denote the set of all maximal ideals of $B$
not containing $e$.
The sets $D_e$  ($e\in B$)
form a basis for the Zariski topology on $\X$, which consists of clopen sets.
If $U$ is a clopen subset of $\X$, then its characteristic function
${\c}_U$ (which is $1$ on $U$ and $0$ on $\X\setminus U$) is trivially an
element of $C_K(\X)$.
If $U=D_e$ for some  $e\in B$, we will denote ${\c}_U$ simply by ${\c}_e$.
Since $\X$, endowed with the Zariski topology, is compact,
a function $\f:\X\to K$ is continuous if and only
if there exist finitely many pairwise orthogonal idempotents $e_i\in B$
with $\wii e_i=1$ such that $\f$ is equal to a linear combination
of the form $\Sii a_i{\c}_{e_i}$ with all $a_i\in K$.
We may, and will, assume without loss of generality that in all such
linear combinations $a_0=0$, $a_1,\dots,a_{n-1}\not=0$, and
$e_1,\dots,e_{n-1}\not=0$; that is,
$D_{e_0}$ is the subset $\f^{-1}(0)$ of $\X$ (which may be the empty set),
while $\bigcup_{i=1}^{n-1} D_{e_i}=\f^{-1}(K\setminus \{0\})$ is the {\it
support} of $\f$, where $D_{e_1},\dots,D_{e_{n-1}}\not=\emptyset$.
It is proved in \cite{as1} that $C_K(\X)$ is a B\'ezout ring with divisibility
theory isomorphic to $B$.

Now we are in a position to construct our B\'ezout ring $R$ with divisibility
theory $S$.
Let $I$ be the ideal of $C_K(\X)$ generated by all elements
${\c}_e$ with $e\in P$.
Hence, $I$ consists of all linear combinations $\sum_{i=0}^{n-1}a_i\c_{e_i}$
where $e_0,\dots,e_{n-1}$ are pairwise orthogonal idempotents 
such that $\wii e_i=1$, $e_1,\dots,e_{n-1}$ are nonzero idempotents
in $P$, and for the coefficients
$a_0,a_1,\dots,a_{n-1}$ we have that $a_0=0$ and
$0\not=a_i\in K$ for all $i>0$ ($i\le n-1$).
Clearly, $I$ is a regular ring without $1$.
It is also easy to see that
the idempotent elements of $I$ are exactly the elements $\c_e$ with $e\in P$.
Furthermore, the principal ideals of $I$ are the ideals
$\c_eI=\c_eC_K(\X)$ with $e\in P$, so
they form a semigroup order-isomorphic to $P$.

Next we want to define an action of $H$ on $I$.
Since $K$ is the field of fractions of $H/Q$, we see that
for each $h\in H$ the element
$\hat{h}=h+Q\,(\in H/Q)$ is in $K$.
Hence the constant function $\hat{h}\c_1$,
where $\c_1=\c_{\X}$ is the constant function with value $1$,
is in the ring $C_K(\X)$. Clearly, 
multiplication by $\hat{h}\c_1$ is an $I$-module endomorphism of $I$,
which is the zero endomorphism if $h\in Q$,
and is an automorphism if $h\in H\setminus Q$,
because in that case $\hat{h}\in K\setminus\{0\}$.
It follows, in particular, that multiplication by $\hat{h}\c_1$ ($h\in H$)
maps every ideal of $I$ into itself.
For each element $h\in H$ we define the action of $h$ on $I$ by
\[
hi:=(\hat{h}\c_1)i\quad \text{for all $i\in I$}.
\]

Our ring $R$ is now defined, as in \cite[Example~3.2]{vas}, on
the product set $H\times I$ with componentwise addition and with multiplication
\begin{equation*}
  (h_1, i_1)(h_2, i_2)=(h_1h_2, h_1i_2+h_2i_1+i_1i_2)
\end{equation*}  
where $h_1i_2$ and $h_2i_1$ are obtained by the action of $H$ on $I$
defined above.
Clearly, $H\to R$, $h\mapsto(h,0)$ is an embedding of the (unital)
ring $H$ in $R$. It is easy to see that $R$ is a commutative
unital ring with unity $1=(1,0)$.
We also have an embedding of the (non-unital) ring $I$ in $R$, namely
$I\to R$, $i\mapsto (0,i)$.
If there is no danger of confusion, we may write $h$ in place of $(h,0)$
and $i$ in place of $(0,i)$.

First we determine the idempotents in $R$.
Assume that $r=(h, g)\in R$ is idempotent.
Then the equality
$(h^2, 2hg+g^2)=r^2=r=(h, g)$ implies that
$h^2=h$, so $h\in \{0, 1\}$ (since $H$ is an
integral domain). Hence,
either $h=0$ and $g^2=g$, or $h=1$ and $g^2=(-g)^2=-g$.
Equivalently, either $r=(0,g)$ is an idempotent in $I$, or
$r=(1,g)=(1,0)-(0,-g)=1-(0,-g)$ where $(0,-g)$ is an idempotent in $I$.

Next we determine the principal ideals of $R$.
Let $r=(h,g)$ be an arbitrary element of $R$, so $h\in H$ and $g\in I$.
Since the divisibility theory of the B\'ezout domain $H$ is isomorphic to
$C\cup \{0\}$, 
there exist $c\in C\cup\{0\}$ and a unit
$u\in H$ such that $uh=c$.
Therefore $rR=urR=(c, ug)R$, so we may assume from now on that
$r=(c, g)$ with  $c\in C\cup\{0\}=S_1\cup S_a\cup\{0\}$ and $g\in I$.
By the description of the principal ideals of $I$ discussed earlier,
there exists 
$e\in P$ such that $gI=\c_eI=\c_eR$.
Consequently, there exist $x, y\in I$ with $xg={\c}_e$ and $g=y{\c}_e$.
If $c\in S_a\cup\{0\}\,(\subseteq Q)$, then $ci=0$ for all $i\in I$, so
both of the elements
$rx=(c,g)(0,x)=(0,gx)=gx=\c_e$ and $\c_ey=g=(0,g)$ are in the ideal $rR$.
Therefore the element $c=(c,0)=(c,g)-(0,g)$, too, belongs to $rR$,
which implies that
$rR=cR\oplus gR$. Similarly, by switching the roles of $g$ and $\c_e$
(along with those of $x$ and $y$), we get that $(c,\c_e)R=cR\oplus \c_eR$.
But $gR=gI=\c_eI=\c_eR$, therefore we conclude that
\begin{equation}
  \label{eq-directsum}
  rR=(c,\c_e)R=cR\oplus \c_eR=cR\oplus gR
  \text{ for some $e\in P$,\quad if $c\in S_a\cup\{0\}$}.
\end{equation}

Now let $c\in S_1$. By the definition of $I$, we have
$g=\sum_{j=0}^{n-1} a_j{\c}_{e_j}$ where $e_0,\dots,e_{n-1}$ are 
pairwise orthogonal idempotents with $\bigwedge_{j=0}^{n-1} e_j=1$,
$0\not=e_1,\dots,e_{n-1}\in P$,
and the coefficients
$a_0,\dots,a_{n-1}$ satisfy the following conditions:
$a_0=0$ and $0\neq a_j\in K$ for every $j>0$ ($j\le n-1$).
Let $e=\bigwedge_{j=1}^{n-1} e_j$. Clearly, $e\in P$.
Furthermore, if $g=0$, then $n=1$, $e_0=1$, $a_0=0$, and $e=0$, while if
$g\not=0$, then we must have $n>1$ and $e\not=0$.
Our aim is to prove that $rR=c(1,-\c_{e'})R$ for some $e'\in P$.
This is clearly true if $g=0$, because then $r=c=c(1,0)=c(1,-\c_0)$.
Therefore we will assume from now on that $g\not=0$, and hence $n>1$.
Notice first that 
for every subscript $\ell>0$ ($\ell\le n-1$),
\begin{equation}
  \label{eq1}
  rR\ni
  r\c_{e_\ell}=(c,g)(0,\c_{e_\ell})
  =\biggl(c,\sum_{j=0}^{n-1}a_j\c_{e_j}\biggr)(0,\c_{e_\ell})
  =\bigl(0,(a_\ell+c)\c_{e_\ell}\bigr)
  =(a_\ell+c)\c_{e_\ell}.
\end{equation}
Let $b_\ell=a_\ell+c$ for each $\ell>0$ ($\ell\le n-1$), and let
\begin{equation*}
  e_{=}=\bigwedge_{\substack{1\le j\le n-1\\b_j=0}}e_j
  \qquad\text{and}\qquad
  e_{\neq}=\bigwedge_{\substack{1\le j\le n-1\\b_j\not=0}}e_j.
\end{equation*}
Then $e_{=},e_{\neq}\in P$, $e_{=}e_{\neq}=0$, and $e_{=}\wedge e_{\neq}=e$.
Hence,
\begin{equation*}
\c_{e_{=}}=\sum_{\substack{1\le j\le n-1\\b_j=0}}\c_{e_j},\quad
\c_{e_{\neq}}=\sum_{\substack{1\le j\le n-1\\b_j\not=0}}\c_{e_j},\quad 
\c_{e_{=}}\c_{e_{\neq}}=\c_0=0,\quad
\text{and}\quad
\c_{e_{=}}+\c_{e_{\neq}}=\c_e.
\end{equation*}
  Let $i=\sum_{j=1}^{n-1}b_j\c_{e_j}=\sum_{j,\,b_j\not=0}b_j\c_{e_j}$.
For every $j$ with $b_j\not=0$ we have that
$\c_{e_j}I=b_j\c_{e_j}I$, therefore $\c_{e_{\neq}}R=\c_{e_{\neq}}I=iI=iR$. 
By \eqref{eq1},  
we also have that $i\in rR$, and hence
\begin{multline*}
  \qquad
  rR\ni
  r-i
  = \biggl(c,\sum_{j=0}^{n-1}a_j\c_{e_j}\biggr)-
  \biggl(0,\sum_{j=1}^{n-1}(a_j+c)\c_{e_j}\biggr)
  = \biggl(c,\sum_{j=1}^{n-1}(-c)\c_{e_j}\biggr)\\
  =c-c\c_e=c(1,-\c_e).
  \qquad
\end{multline*}  
Thus, both principal ideals $\c_{e_{\neq}}R=iR$ and
$c(1,-\c_e)R$ are contained in $rR$.
Their product is $\{0\}$, because
$\c_{e_{\neq}}(1,-\c_{e})=(0,\c_{e_{\neq}})(1,-\c_{e})=(0,\c_{e_{\neq}}-\c_{e}\c_{e_{\neq}})
=(0,0)=0$, and their sum $c(1,-\c_e)R+\c_{e_{\neq}}R=c(1,-\c_e)R+iR$ is $rR$,
because $c(1,-\c_e)+i=r$. This proves that $rR=c(1,-\c_e)R\oplus\c_{e_{\neq}}R$.
Finally, we claim that
\begin{equation*}
  rR=c(1,-\c_e)R\oplus\c_{e_{\neq}}R=c(1,-\c_{e_{=}})R.
\end{equation*}
For the proof of the second equality,
the inclusion $\subseteq$ follows by observing that
$c(1,-\c_e)=c(1,-\c_{e_{=}})(1,-\c_e)$ and
$\c_{e_{\neq}}=c(1,-\c_{e_{=}})(c^{-1}\c_{e_{\neq}})$, while the
reverse inclusion $\supseteq$ is a consequence of the equality
$c(1,-\c_{e_{=}})=c(1,-\c_e)+c\c_{e_{\neq}}$.

To summarize our discussion in the last two paragraphs, we have proved that 
the set of principal ideals of $R$ is
\begin{equation}
  \label{eq_principIdeals}
\mathcal{I}_{\textrm{princ}}(R)=  
\{c(1,-\c_e)R: c\in S_1,e\in P\}\cup \{(d,\c_e )R : d\in S_a\cup\{0\},e\in P\}.
\end{equation}
Next we will show that $R$ is a B\'ezout ring, that is, every
finitely generated ideal of $R$ is principal.
It suffices to prove that for any two principal ideals $J_1,J_2$ of $R$,
the ideal $J=J_1+J_2$ is principal.
We will use the description of $\mathcal{I}_{\textrm{princ}}(R)$
in \eqref{eq_principIdeals},
and will distinguish three cases according to which of the two subsets
of $\mathcal{I}_{\textrm{princ}}(R)$ indicated in
\eqref{eq_principIdeals} $J_1$ and $J_2$ belong to.

Assume first that both ideals $J_i$ ($i=1,2$)
have the form $J_i=c_i(1,-\c_{e_i})R$ where $c_i\in S_1$ and $e_i\in P$.
Then $c_1(1,-\c_{e_1})\c_{e_2}=c_1(\c_{e_2}-\c_{e_1e_2})\in J_1\cap I$, so since
the action of $c_1\in S_1$ on $I$ is an $I$-module automorphism which maps
every ideal of $I$ into itself, we get that $\c_{e_2}-\c_{e_1e_2}\in J_1$.
Hence,
$c_2(1, -\c_{e_1e_2})=c_2(\c_{e_2}-\c_{e_1e_2})+c_2(1,-\c_{e_2})\in J_1+J_2=J$.
By switching the roles of $J_1$ and $J_2$, we get similarly that
$c_1(1, -\c_{e_1e_2})\in J$.
We want to show that for the element $c=c_1\wedge c_2$, which
belongs to $S_1$
by Proposition~\ref{prop-bezout2}(2), we also have
$c(1, -\c_{e_1e_2})\in J$.
If $c_1=c_2$, then $c=c_1=c_2$, so our claim $c(1, -\c_{e_1e_2})\in J$
trivially follows from $c_i(1, -\c_{e_1e_2})\in J$ ($i=1,2$).
So, assume from now on that $c_1\neq c_2$. 
Using the hypernormality of $S$
(see Definition~\ref{basic}(4)) we get that there exist 
elements $w_1, w_2\in S$ such that $c_1=cw_1, c_2=cw_2$, and
$w_1\wedge w_2=1$.
The equalities $c_1=cw_1$ and $c_2=cw_2$ imply that
$c, w_1, w_2$ are all in $S_1$.
Since $S_1$ is cancellative (see Proposition~\ref{prop-bezout2}(2))
and $c_1\neq c_2$, we get that $w_1\neq w_2$.
Hence, by the definition of $H$, the element
$h=w_1+w_2\in H$ has an inverse in $H$.
Therefore, the following calculation shows that $c(1,-\c_{e_1e_2})\in J$:
\[
c(1,-\c_{e_1e_2})=h^{-1}hc(1,-\c_{e_1e_2})
=h^{-1}\bigl(c_1(1,-\c_{e_1e_2}) + c_2(1,-\c_{e_1e_2})\bigr)\in J.
\]
In fact, we also have that $J=c(1,-\c_{e_1e_2})R$, 
because the equality
$c_i(1,-\c_{e_i})=w_ic(1,-\c_{e_1e_2})(1, -\c_{e_i})$ implies that
$J_i\subseteq c(1,-\c_{e_1e_2})R$ for $i=1,2$.
This finishes the proof that $J$ is a principal ideal in this case.

Next, let us assume that $J_1=c_1(1,-\c_{e_1})R$ and
$J_2=(d_2, \c_{e_2})R$ where $c_1\in S_1$, $d_2\in S_a\cup\{0\}$, and
$e_1,e_2\in P$.
Recall from \eqref{eq-directsum} that $J_2=d_2R\oplus \c_{e_2}R$, and hence
$d_2=(d_2,0)\in J_2$ and $\c_{e_2}\in J_2$.
Let $e$ be the relative complement of $e_1e_2$ in $e_1$; that is, let
$e\in B$ be such that $e_1e_2\wedge e=e_1$ and $e_1e_2\vee e=0$.
As before, let $c=c_1\wedge d_2$, and
notice that $c\in S_1$, because by Proposition~\ref{prop-bezout2}(1),
$c^\perp=c_1^\perp\cap d_2^\perp=\{0\}\cap a^\perp=\{0\}$.
Furthermore, let us use the hypernormality of
the B\'ezout monoid $C\cup\{0\}$
to obtain elements $w_1,w_2\in C\cup\{0\}$ such that $c_1=cw_1$, $d_2=cw_2$, and
$w_1\wedge w_2=1$. 
We note that if $d_2\in S_a$, then
$w_2\in S_a$, because the equality $d_2=cw_2$ with
$d_2\in S_a$, $c\in S_1$ implies that $w_2\notin S_1\cup\{0\}$,
hence it must be that $w_2\in S_a$.
If, in turn, $d_2=0$, then the equality $0=d_2=cw_2$ with
$c\in S_1$ forces $w_2=0$.
Now let $h=w_1+w_2\in H$.
If $d_2=0$, and hence $w_2=0$, then $w_1\wedge w_2=1$ implies that $w_1=1$,
so $h=1$.
If, in turn, $d_2\in S_a$, then $w_1\in S_1$, $w_2\in S_a$ imply
that $w_1,w_2$ are distinct elements of $C$.
Therefore
the definition of $H$ yields that in both cases
$h$ has an inverse in $H$.
Therefore, the following calculation shows that $c(1,-\c_{e_1})\in J$:
\begin{multline*}
c(1, -\c_{e_1})
=h^{-1}hc(1, -\c_{e_1})
=h^{-1}\bigl(c_1(1,-\c_{e_1})+d_2(1,-\c_{e_1})\bigr)\\
=h^{-1}\bigl(c_1(1,-\c_{e_1})+(d_2,0)\bigr)
\in J_1+J_2=J,
\end{multline*}
where the third equality holds, because $d_2\in S_a$ acts on $I$ by the zero
map, and hence $d_2\c_{e_1}=0$.
Thus,
\[
c(1,-\c_e)=c(1, -\c_{e_1})+c\c_{e_1e_2}=c(1, -\c_{e_1})+\c_{e_2}c\c_{e_1}
\in J+J_2=J,
\]
which implies that
$c(1,-\c_e)R\subseteq J$.
To prove that $J=c(1,-\c_e)R$
we need to establish that $J_1,J_2\subseteq c(1,-\c_e)R$.
For $J_1$, this follows from the equality
$c_1(1, -\c_{e_1})=c(1, -\c_e)d_1(1,-\c_{e_1})$.
For $J_2=d_2R\oplus\c_{e_2}R$, the desired conclusion follows
from the equalities
$\c_{e_2}=c(1, -\c_e)(c^{-1}\c_{e_2})$ ($c^{-1}\in K$)
and
$d_2=w_2c=(w_2c,0)=\bigl(w_2c,w_2c(-\c_e)\bigr)=w_2c(1,-\c_e)$,
where the third $=$ is valid, since $w_2\in S_a$ acts on $I$ by the zero map,
forcing $w_2\c_e=0$.
This proves that $J$ is a principal ideal in this case as well.

Finally, assume that both ideals $J_i$ ($i=1,2$) have the form
$J_i=(d_i, \c_{e_i})R=d_iR\oplus \c_{e_1}R$ where $d_1,d_2\in S_a\cup\{0\}$
and $e_1,e_2\in P$.
Then it is easy to see that
\[
J=(d_1R+d_2R)\oplus(\c_{e_1}R+\c_{e_2}R)
=(d_1\wedge d_2)R\oplus\c_{e_1\vee e_2}R
=(d_1\wedge d_2,\c_{e_1\vee e_2})R,
\]
where $d_1\wedge d_2\in S_a\cup\{0\}$, so
$J$ is a principal ideal in this final case as well.
Thus, the proof that $R$ is a B\'ezout ring is now complete.

Finally, we prove that the divisibility theory of $R$, i.e.,
the B\'ezout monoid $\mathcal{I}_{\textrm{princ}}(R)$
with ideal multiplication, is isomorphic to $S$.
We will use the description of $\mathcal{I}_{\textrm{princ}}(R)$
in \eqref{eq_principIdeals} and the description of $S$ in
Proposition~\ref{prop-vas}(4)--(5).
Statements (4) and (1) of Proposition~\ref{prop-vas},
combined with \eqref{eq_principIdeals}, 
imply that the following assignment yields a well-defined and surjective
function $S\to \mathcal{I}_{\textrm{princ}}(R)$:
\begin{align*}
  \Phi\colon S & {}\to \mathcal{I}_{\textrm{princ}}(R),\\
  cf & {}\mapsto c(1,-\c_e)R,\ \ \text{if $c\in S_1$ and
    $f$ is the orthogonal complement of $e\in P$},\\
  d\wedge e & {}\mapsto (d,\c_e)R, \kern22pt 
    \text{if $d\in S_a\cup\{0\}$ and $e\in P$}\hfill. 
\end{align*}
To see that $\Phi$ is also injective, we need to argue that the
principal ideals shown in \eqref{eq_principIdeals} with different generators
are different ideals. This follows from the following observations:
\begin{itemize}
\item
  If $c\in S_1$ and $e\in P$, then $c(1,-\c_e)R+I=cH+I$, and
  the annihilator ideal of
  $c(1,-\c_e)$ in $R$ is $\c_eR$.
\item
  If $d\in S_a\cup\{0\}$ and $e\in P$, then
  $(d,\c_e)R=dR\oplus\c_eR$ (see \eqref{eq-directsum}), therefore
  $(d,\c_e)R\cap H=dH$ and $(d,\c_e)R\cap I=\c_eH$.
\end{itemize}  
It remains to show that $\Phi$ is a monoid isomorphism, that is,
$\Phi(s_1s_2)=\Phi(s_1)\Phi(s_2)$ holds for all $s_1,s_2\in S$.
Checking this is fairly straightforward, by distinguishing
three cases according to which of the two
disjoint subsets $S_1(B\setminus P)$ and $(S_a\cup\{0\})\wedge P$
of $S$ (see Proposition~\ref{prop-vas}(4))
the elements $s_1,s_2$ belong to. In Case~1 we will assume that  
$s_1,s_2\in S_1(B\setminus P)$, that is,
both elements $s_i$ ($i=1,2$) have the form $s_i=c_if_i$
where $c_i\in S_1$, $f_i\in B\setminus P$, so for the orthogonal
complement $e_i$ of $f_i$ we have $e_i\in P$;
in Case~2, $s_1,s_2\in (S_a\cup\{0\})\wedge P$, that is,
both elements $s_i$ ($i=1,2$) have the form $s_i=d_i\wedge e_i$
where $d_i\in S_a\cup\{0\}$ and $e_i\in P$; finally, in Case~3,
$s_1\in S_1(B\setminus P)$ and $s_2\in (S_a\cup\{0\})\wedge P$, that is,
$s_1=c_1f_1$ and $s_2=d_2\wedge e_2$ where
$c_1\in S_1$, $d_2\in S_a\cup\{0\}$, 
$f_1\in B\setminus P$, and $e_2\in P$, so for the orthogonal
complement $e_1$ of $f_1$ we have $e_1\in P$.
The following calculations prove the equality
$\Phi(s_1s_2)=\Phi(s_1)\Phi(s_2)$ in all three cases.
In each case the first and second to last
equalities follow from the definition of $\Phi$, while
the last equality follows from Proposition~\ref{prop-vas}(5).
So, in Case~1, we have
\begin{multline*}
\Phi(c_1f_1)\Phi(c_2f_2)=c_1(1,-\c_{e_1})Rc_2(1,-\c_{e_2})R
=c_1c_2(1,-\c_{e_1\wedge e_2})R\\
=\Phi\bigl(c_1c_2(f_1\vee f_2)\bigr)
=\Phi\bigl((c_1f_1)(c_2f_2)\bigr).
\end{multline*}
In Case~2, we have $d_1\c_{e_2}=0=d_2\c_{e_1}$,
because $d_1,d_2\in S_a\cup\{0\}$ act on $I$
by the zero map. Therefore,
\begin{multline*}
\Phi(d_1\wedge e_1)\Phi(d_2\wedge e_2)=(d_1,\c_{e_1})R(d_2,\c_{e_2})R
=(d_1d_2,d_1\c_{e_2}+d_2\c_{e_1}+\c_{e_1}\c_{e_2})R\\
=(d_1d_2,\c_{e_1\vee e_2})R
=\Phi\bigl(d_1d_2\wedge(e_1\vee e_2)\bigr)
=\Phi\bigl((d_1\wedge e_1)(d_2\wedge e_2)\bigr).
\end{multline*}
Finally, in Case~3, we have $d_2\c_{e_1}=0$ for the same reason as in the
preceding case. Hence, if $e$ denotes the relative complement of
$e_1e_2$ in $e_2$, then
\begin{multline*}
\Phi(c_1f_1)\Phi(d_2\wedge e_2)=c_1(1,-\c_{e_1})R(d_2,\c_{e_2})R\\
=c_1(d_2,\c_{e_2}-d_2\c_{e_1}-\c_{e_1}\c_{e_2})R
=c_1(d_2,\c_{e_2}-\c_{e_1}\c_{e_2})R=c_1(d_2,\c_{e})R\\
=c_1(d_2R\oplus \c_eR)=c_1d_2R\oplus c_1\c_eR=
c_1d_2R\oplus \c_eR
=(c_1d_2,\c_{e})R\\
=\Phi\bigl(c_1d_2\wedge e\bigr)
=\Phi\bigl((c_1f_1)(d_2\wedge e_2)\bigr),
\end{multline*}
where the first and fourth $=$'s on line~3 are based on \eqref{eq-directsum}
(and also on the fact that $d_2,c_1d_2\in S_a\cup\{0\}$), and
the third equality on line~3
holds, because $c_1\in S_1$ acts on $I$ by an $I$-module
automorphism which maps every ideal of $I$ into itself. 

This completes the proof of Theorem~\ref{vas2}.
\end{proof}

\bibliographystyle{plain}

\end{document}